\documentclass[12pt,a4paper]{article}
\usepackage{latexsym,amsfonts,amsmath,graphics,amsthm}
\usepackage{epsfig}
\usepackage{mathrsfs}

\newtheorem{theorem}{Theorem}
\newtheorem{lemma}{Lemma}
\newtheorem{corollary}{Corollary}

\newtheorem{remark}{Remark}

\newcommand{\abs}[1]{\left\vert#1\right\vert}

\newcommand{\To}{\rightarrow}

\newcommand{\bsa}{\boldsymbol{a}}

\newcommand{\bsc}{\boldsymbol{c}}
\newcommand{\bsn}{\boldsymbol{n}}

\newcommand{\bsx}{\boldsymbol{x}}

\newcommand{\bsb}{\boldsymbol{b}}

\newcommand{\bsy}{\boldsymbol{y}}

\newcommand{\icomp}{\mathtt{i}}

\newcommand{\rd}{\,\mathrm{d}}

\newcommand{\NN}{\mathbb{N}}
\newcommand{\ZZ}{\mathbb{Z}}

\newcommand{\FF}{\mathbb{F}}

\newcommand{\RR}{{\mathbb R}}

\newcommand{\cP}{{\mathscr P}}
\newcommand{\ee}{\mathrm{e}}

\newcommand{\bsalpha}{\boldsymbol{\alpha}}
\makeatletter
\newcommand{\rdots}{\mathinner{\mkern1mu\lower-1\p@\vbox{\kern7\p@\hbox{.}}
\mkern2mu \raise4\p@\hbox{.}\mkern2mu\raise7\p@\hbox{.}\mkern1mu}}
\makeatother

\allowdisplaybreaks

\begin{document}
\title{Discrepancy estimates for index-transformed uniformly distributed sequences}

\author{Peter Kritzer, 
Gerhard Larcher,   
Friedrich Pillichshammer\thanks{The authors gratefully acknowledge the support of the Austrian Science Fund (FWF), Projects P23389-N18 and F5506-N26 (Kritzer), 
P21943 and F5507-N26 (Larcher), S9609 and F5509-N26 (Pillichshammer). 
Project S9609 is part of the Austrian Research Network "Analytic Combinatorics and Probabilistic Number Theory". 
Projects F5506-N26, F5507-N26, F5509-N26 are parts of the special research program "Quasi-Monte Carlo Methods: Theory and Applications".}}

\date{}
\maketitle

\begin{abstract}
\noindent In this paper we show discrepancy bounds for index-transformed uniformly distributed sequences. 
From a general result we deduce very tight lower and upper bounds on the discrepancy of index-transformed van der Corput-, Halton-, and $(t,s)$-sequences indexed by the sum-of-digits function. We also analyze the 
discrepancy of sequences indexed by other functions, such as, e.g., $\lfloor n^{\alpha}\rfloor$ with $0 < \alpha < 1$.
\end{abstract}



\section{Introduction}

A sequence $(\bsy_n)_{n\ge 0}$ in the unit-cube $[0,1)^s$ is said to be {\it uniformly distributed modulo one} if for all intervals $[\bsa,\bsb) \subseteq [0,1)^s$ it is true that
\begin{equation}\label{udm1}
\lim_{N \rightarrow \infty} \frac{\#\{n\, : \, 0 \le n <N, \bsy_n \in   [\bsa,\bsb)\}}{N}=\mathrm{vol} ([\bsa,\bsb)).
\end{equation}
A quantitative version of \eqref{udm1} can be stated in terms of discrepancy. For an infinite 
sequence $(\bsy_n)_{n\ge 0}$ in $[0,1)^s$ its {\it discrepancy} is 
defined as 
$$D_N((\bsy_n)_{n\ge 0}):=\sup_{[\bsa,\bsb)
\subseteq [0,1)^s}\left|\frac{\#\{n\, : \, 0 \le n <N, \bsy_n \in   [\bsa,\bsb)\}}{N}-\mathrm{vol} ([\bsa,\bsb))\right|,$$ 
where the supremum is extended over all sub-intervals $[\bsa,\bsb)$ of $[0,1)^s$. For a given finite sequence $X=(\bsx_1,\ldots,\bsx_M)$ we write
$D_M(X)$ for the discrepancy of $X$ with the obvious adaptions in the above definition. An infinite sequence is uniformly distributed 
modulo one if and only if its discrepancy tends to zero as $N$ goes to infinity. However, convergence of the discrepancy to zero 
cannot take place arbitrarily fast. It follows from a result of Roth \cite{roth1} that for any infinite sequence
$(\bsy_n)_{n\ge 0}$ in $[0,1)^s$ we have $N D_{N}((\bsy_n)_{n\ge 0})\ge c_s (\log N)^{s/2}$ for infinitely many values of 
$N \in \NN$ (by $\NN$ we denote the set of positive integers,
and we put $\NN_0:=\NN\cup\{0\}$). An improvement of this bound can be obtained from \cite{BLV08}. 
For the special case $s=1$, Schmidt~\cite{Schm72distrib} (see also \cite{bej}) 
showed that for any infinite sequence $(y_n)_{n\ge 0}$ in $[0,1)$ we have $N D_{N}((y_n)_{n\ge 0})\ge\frac{\log N}{66\log 4}$ for infinitely many values of 
$N \in \NN$. This result is best possible with respect to the order of magnitude in $N$. An excellent introduction to this topic can be found in the book of Kuipers and Niederreiter \cite{kuinie} (see also \cite{DP10,dt,matou,niesiam}). 

Well known examples of uniformly distributed sequences are $(n\bsalpha)$-sequences (also called Kronecker-sequences, see \cite{dt,kuinie}), van der Corput-sequences and their multivariate analogues called Halton-sequences (see \cite{DP10,huawang,kuinie,niesiam}), as well as (digital) $(t,s)$-sequences (see \cite{DP10,niesiam}). 

\bigskip

In recent years, also the distribution properties of index-transformed uniformly distributed sequences have been studied, especially for the examples mentioned above. In this paper, we mean by an index-transformed sequence of a sequence $(x_n)_{n \ge 0}$ a sequence $(x_{f(n)})_{n \ge 0}$, where $f:\NN_0 \rightarrow \NN_0$. Note that $(x_{f(n)})_{n \ge 0}$ is in general no subsequence of $(x_n)_{n \ge 0}$ since we do {\it not} require that $f$ is strictly increasing.

For instance, the distribution properties of index-transformed
Kronecker-sequences indexed by the sum-of-digits function were studied in \cite{C80,DL01,TTa,TT}. For this special case, very precise results 
can be found in \cite{DL01}. In \cite{D98} the well-distribution of index-transformed Kronecker-sequences indexed by $q$-additive functions is considered. 
Furthermore, in \cite{P12} a discrepancy bound for van der Corput-sequences in bases of the form $b=5^\ell$, $\ell\in\NN$, indexed by Fibonacci numbers is shown. 
The papers \cite{HN,HKLP,P12} deal with index-transformed van der Corput-, Halton-, and $(t,s)$-sequences. 

In this paper we are specifically interested in discrepancy bounds for sequences indexed by the $q$-ary sum-of-digits function and related functions 
and, furthermore, for sequences indexed by ``moderately'' monotonically increasing sequences, as for example $\lfloor n^{\alpha}\rfloor$ with $0 < \alpha < 1$. 
For an integer $q \ge 2$ and $n \in \NN_0$ with base $q$ expansion $n=r_0+r_1 q+r_2 q^2+\cdots$ the \textit{$q$-ary sum-of-digits function} is defined by $s_q(n):=r_0+r_1+r_2+\cdots$.

Previously, it has been shown in \cite{HKLP} that the sequence $(\bsx_{s_q(n)})_{n\ge 0}$, indexed by the $q$-ary sum-of-digits function, where $(\bsx_n)_{n \ge 0}$ denotes the Halton-sequence in 
co-prime bases $b_1,\ldots,b_s$ is uniformly distributed modulo one. The proof of this result is due to the fact that the sequence generated by the $q$-ary sum-of-digits function is 
uniformly distributed in $\ZZ$, see, for example, \cite{gel,PS02}. In this paper we provide very tight lower and upper bounds on the discrepancy of index-transformed van der Corput-, 
Halton-, and $(t,s)$-sequences indexed by the sum-of-digits function. 

\bigskip

This paper is structured as follows. In Section~\ref{secdef}, we provide basic definitions and notation used throughout the subsequent sections.
In Section~\ref{secgeneral}, we prove a general theorem (Theorem~\ref{thmgeneral}) which will be of great importance in discussing sequences indexed by the sum-of-digits function. In Section~\ref{secsod} 
we present a concrete application of Theorem~\ref{thmgeneral} which leads to the aforementioned tight 
bounds on the discrepancy of Halton- and $(t,s)$-sequences indexed by $s_q(n)$. 
Furthermore, we discuss a refinement of these results for van der Corput-sequences. 
Finally, in Section~\ref{secother}, we deal with discrepancy bounds for sequences which 
are obtained by certain moderately increasing index sequences, such as, e.g., $\lfloor n^{\alpha}\rfloor$ with $0 < \alpha < 1$.

\section{Notation and basic definitions}\label{secdef}

We first outline the definitions of the sequences studied in this paper, namely van der Corput-, Halton-, and $(t,s)$-sequences.

\bigskip

Let $b \ge 2$ be an integer. A {\it van der Corput-sequence $(x_n)_{n\ge 0}$ in base $b$} is defined by $x_n=\varphi_b(n)$,
where for $n \in \NN_0$, with base $b$ expansion $n=a_0+a_1b+a_2 b^2+\cdots$, the 
so-called {\it radical inverse function} $\varphi_b: \NN_0 \rightarrow [0,1)$ is 
defined by 
$$\varphi_b(n):=\frac{a_0}{b}+\frac{a_1}{b^2}+\frac{a_2}{b^3}+\cdots.$$ 
It is well known that for any base $b \ge 2$ the corresponding van der Corput-sequence is uniformly 
distributed modulo one and that $N D_N((x_n)_{n\ge 0}) = O(\log N)$, see, for example, \cite{befa,DP10,kuinie}. 

If we choose co-prime integers $b_1,\ldots,b_s\ge 2$, then $s$ one-dimensional van der Corput-sequences can be combined to an 
$s$-dimensional uniformly distributed 
sequence with points $\bsx_n:=(\varphi_{b_1}(n),\ldots,\varphi_{b_s}(n))$
for $n\in\NN_0$. This sequence is called a {\it Halton-sequence} 
and it is known that its discrepancy is of order $(\log N)^s /N$, see \cite{A04,DP10,fau80,FL,halton,huawang,meijer,niesiam}. Note 
that Halton-sequences are a
direct generalization of van der Corput-sequences, so van der Corput-sequences can be viewed as one-dimensional Halton-sequences, and indeed
Halton-sequences are sometimes also referred to as van der Corput-Halton-sequences (see, e.g., \cite{kuinie}). However, as there
will be results in this paper which only hold for the one-dimensional case, 
it will be useful to explicitly distinguish van der Corput-sequences (which we use for the
one-dimensional variant) from Halton-sequences (which we use for the multi-dimensional variant).

\bigskip

Another type of sequences we will be concerned with in this paper are $(t,s)$-sequences, 
for the definition of which we need the definition of elementary intervals and $(t,m,s)$-nets in base $b$. 

For an integer $b\ge2$, an {\it elementary interval} in base $b$
is an interval of the form $\prod_{i=1}^s [a_i b^{-d_i},(a_i +1) b^{-d_i})\subseteq [0,1)^s$, where $a_i, d_i$ are non-negative integers with
$0\le a_i< b^{d_i}$ for $1\le i\le s$. 

Let $t,m$, with $0\le t\le m$, be integers. Then a {\it $(t,m,s)$-net in base $b$} is a point set $(\bsy_n)_{n=0}^{b^m-1}$ in $[0,1)^s$ such that
any elementary interval in base $b$ of volume $b^{t-m}$ contains exactly $b^t$ of the $\bsy_n$. 

Furthermore, we call an infinite sequence $(\bsx_n)_{n\ge 0}$ a {\it $(t,s)$-sequence in base $b$} if the subsequence $(\bsx_n)_{n=kb^m}^{(k+1)b^m -1}$ is 
a $(t,m,s)$-net in base $b$ for all integers $k\ge 0$ and $m\ge t$. It is known (see, e.g., \cite{DP10, N87, niesiam}) 
that a $(t,s)$-sequence is particularly evenly distributed if the value of $t$ is small. In particular,
it can be shown that the discrepancy of a $(t,s)$-sequence in base $b$ is of order $b^t (\log N)^s /N$, see, e.g., \cite{DP10,N87, niesiam}.

\bigskip

A very important sub-class of $(t,s)$-sequences is that of digital $(t,s)$-sequences, which are defined over algebraic structures like finite fields or rings. 
For the sake of simplicity, we restrict ourselves to digital sequences over finite fields $\FF_p$ of prime order $p$. Again for the sake of simplicity we do not
distinguish, here and later on, between elements in $\FF_p$
and the set of integers $\{0,1,\ldots,p-1\}$ (equipped with
arithmetic operations modulo $p$). 

For a vector $\bsc=(c_1,c_2,\ldots)\in \FF_p^{\infty}$ and
for $m \in \NN$  we denote the vector in $\FF_p^m$
consisting of the first $m$ components of $\bsc$ by $\bsc(m)$,
i.e., $\bsc(m)=  (c_1,\ldots ,c_m)$. Moreover, for an $\NN
\times \NN$ matrix $C$ over $\FF_p$ and for $m \in
\NN$ we denote by $C(m)$ the left upper $m \times m$
submatrix of $C$.

For $s \in \NN$ and $t \in \NN_0$, choose $\NN
\times \NN$ matrices $C_1,\ldots,C_s$ over $\FF_p$ with
the following property. For every $m \in \NN$, $m \ge t$, and
all $d_1,\ldots,d_s \in \NN_0$ with $d_1+\cdots +d_s= m-t$,
the vectors
\[\bsc_1^{\,(1)}(m),\ldots,\bsc_{d_1}^{\,(1)}(m),\ldots,\bsc_1^{\,(s)}(m),\ldots,\bsc_{d_s}^{\,(s)}(m)\] are linearly independent in $\FF_p^m$. Here $\bsc_i^{\,(j)}$ is the $i$-th
row vector of the matrix $C_j$.

For $n \in \NN_0$ let $n=n_0+n_1 p+ n_2 p^2+\cdots$ be the base
$p$ representation of $n$. For every index $1 \le j \le s$
multiply the digit vector $\bsn=(n_0,n_1,\ldots)^{\top}$ by the
matrix $C_j$,
$$C_j \cdot\bsn=:(x_{n,j}(1),x_{n,j}(2),\ldots)^{\top}$$ (note that the matrix-vector multiplication is performed over $\FF_p$), and set
$$x_n^{(j)}:=\frac{x_{n,j}(1)}{p}+\frac{x_{n,j}(2)}{p^2}+\cdots.$$
Finally set $\bsx_n:=(x_n^{(1)},\ldots,x_n^{(s)})$. A sequence
$(\bsx_n)_{n\ge 0}$ constructed in this way is called a {\it
digital} $(t,s)$-{\it sequence over} $\FF_p$. The matrices
$C_1,\ldots ,C_s$ are called the {\it generator matrices} of the
sequence.

To guarantee that the points $\bsx_n$ lie in $[0,1)^s$ (and not
just in $[0,1]^s$) we assume that for each $1 \le j \le s$ and $w
\ge 0$ we have $c_{v,w}^{(j)}=0$ for all sufficiently large $v$,
where $c_{v,w}^{(j)}$ are the entries of the matrix $C_j$ (see \cite[p.72, condition (S6)]{niesiam} for more information).

\bigskip

Throughout the paper we use the following notation. For functions $f,g:\NN \rightarrow \RR$, where $f \ge 0$, we write $g(n)=O(f(n))$ or $g(n) \ll f(n)$, 
if there exists a $C>0$ such that $|g(n)| \le C f(n)$ for all sufficiently large $n \in \NN$. 
If we would like to stress that the quantity $C$ may also depend on other variables than $n$, say $\alpha_1,\ldots,\alpha_w$, 
which will be indicated by writing $\ll_{\alpha_1,\ldots,\alpha_w}$.

\section{A general theorem}\label{secgeneral}

In this section we present a general result for the discrepancy of sequences of the form $(\bsx_{g(n)})_{n \ge 0}$, for a particular class of functions $g: \NN_0 \rightarrow \NN_0$. 
Here and in the following, a sequence $(a_k)_{k\in \NN_0}$ is called {\it unimodal} if the sequence $(a_{k+1}-a_{k})_{k \in \NN_0}$ has exactly one change of sign.

Furthermore, we need the concept of the so-called \textit{uniform discrepancy} of a sequence. The uniform discrepancy of 
a sequence $(\bsx_n)_{n \ge 0}$ in $[0,1)^s$ is defined as 
$$\widetilde{D}_N((\bsx_n)_{n \ge 0}):= \sup_{k \in \NN_0} D_N((\bsx_{n+k})_{n \ge 0}).$$
 
\begin{theorem}\label{thmgeneral}
Let $(\bsx_n)_{n \ge 0}$ be an $s$-dimensional sequence with uniform discrepancy $\widetilde{D}_N=\widetilde{D}_N((\bsx_n)_{n \ge 0})$, 
and let $f:\NN_0\To\RR$ be a non-decreasing function such that $N\widetilde{D}_N\le f(N)$ for $N\in\NN_0$.

Let $g:\NN_0\rightarrow \NN_0$. Furthermore, let $(N_j)_{j \ge 0}$ be a strictly increasing sequence
in $\NN$ with $1=N_0$, and assume that $(N_j)_{j \ge 0}$ is a divisibility chain, i.e., $N_0| N_1$, $N_1 | N_2$, $N_2|N_3$, etc.  Define, for $k\in\NN_0$, 
$$G_{A,j}(k):=\#\{n \, : \,  A N_j \le n < (A+1)N_j, g(n)=k\}.$$
Then the following two assertions hold.
\begin{enumerate}
 \item For $N\in\NN$ with $N_d\le N< N_{d+1}$ we have $N D_N((\bsx_{g(n)})_{n \ge 0}) \ge \max_{k \in \NN_0} G_{0,d}(k)$.
 \item Assume that $G_{A,j}(k)$ is unimodal in $k$ for all $j\in \NN_0$ and all $A \in \NN_0$, and put 
$$G_j:=\max_{k,A \in \NN_0} G_{A,j}(k) \;\mbox{ for }\; j \in \NN_0.$$  

For $j \in \NN_0$ and $A \in \NN_0$ let $$v_{A,j}:=\#\{k\in\NN_0 \, : \, g(n)=k\mbox{ for } A N_j \le n < (A+1)N_j\}$$ and put $$v_j:=\max_{A \in \NN_0} v_{A,j}.$$

Then for $N \in \NN$ with $N_d \le N < N_{d+1}$ we have 
$$
N D_N((\bsx_{g(n)})_{n \ge 0}) \le \sum_{j=0}^d \frac{N_{j+1}}{N_j}  G_j  f(v_j).
$$ 
\end{enumerate}
\end{theorem}
\begin{proof}
\begin{enumerate}

\item To show the lower bound choose a non-negative integer $\kappa$ such that  $\widetilde{G}_d=G_{0,d}(\kappa)=\max_{k \in \NN_0}G_{0,d}(k)$. 
Then the number of $n \in \{0,\ldots,N-1\}$ such that $\bsx_{g(n)}=\bsx_{\kappa}$ is at least $\widetilde{G}_d$ and hence, 
with an arbitrarily small interval containing $\bsx_{\kappa}$ we obtain $$D_N( (\bsx_{g(n)})_{n\ge 0}) \ge \frac{\widetilde{G}_d}{N}.$$

\item To prove the upper bound let 
$$N=a_d N_d+a_{d-1} N_{d-1}+\cdots+a_0 N_0,$$ 
with $a_j\in\NN_0$ and
$$a_j \le \frac{N_{j+1}}{N_j};\;\mbox{ for }\;\; j \in \{0,\ldots,d\}.$$ 
For $j \in \{0,\ldots,d\}$ and $\ell\in\{0,\,\ldots,a_j-1\}$  we consider the sequence 
$$X_{j,\ell}:=(\bsx_{g(A N_j+k)})_{k=0}^{N_j -1}$$ where $A N_j:= a_d N_d+\cdots+a_{j+1}N_{j+1}+ \ell N_j$ (strictly speaking, $A=A(j,\ell)$). 

Since $G_{A,j}$ is unimodal we may assume that for $A N_j \le n < (A+1)N_j$ the function $g(n)$ attains the values 
$$w,w+1,\ldots,w+v,$$ 
for some $w\in\NN_0$ and some integer $v=v_{A,j} \le v(j)$

Assume that the value $w+u_1$ with $0 \le u_1 \le v$ is attained most often, the value $w+u_2$ with $0 \le u_2 \le v$ 
is attained second most often, etc. \ldots, and $w+u_v$ with $0 \le u_v \le v$ (indeed, $u_v \in \{0,v\}$) is attained least often. 
If $w+u_r$ and $w+u_{r+1}$ are both attained the same number of times, then the order of them is of no relevance.

If we consider the sequence $X_{j,\ell}$ as a multi-set (i.e., multiplicity of the elements is relevant, but their order is not), then we can decompose
$X_{j,\ell}$ into
\begin{center}
\begin{tabular}{lll}
$G_{A,j}(w+u_1)-G_{A,j}(w+u_2)$ & times & $\{\bsx_{w+u_1}\}$\\
$G_{A,j}(w+u_2)-G_{A,j}(w+u_3)$ & times & $\{\bsx_{w+u_1},\bsx_{w+u_2}\}$\\
$G_{A,j}(w+u_3)-G_{A,j}(w+u_4)$ & times & $\{\bsx_{w+u_1},\bsx_{w+u_2},\bsx_{w+u_3}\}$\\
\ldots \\
$G_{A,j}(w+u_{v-1})-G_{A,j}(w+u_{v})$ & times & $\{\bsx_{w+u_1},\bsx_{w+u_2},\ldots,\bsx_{w+u_{v-1}}\}$\\
$G_{A,j}(w+u_v)-G_{A,j}(w+u_{v+1})$ & times & $\{\bsx_{w+u_1},\bsx_{w+u_2},\ldots,\bsx_{w+u_v}\},$
\end{tabular}
\end{center}
where we formally set $G_{A,j}(w+u_{v+1}):=0$.
Note that because of the unimodality of $G_{A,j}(k)$, for $r\in\{1,\ldots,v\}$, the sequence $\bsx_{w+u_1},\bsx_{w+u_2},\ldots,\bsx_{w+u_r}$ 
is a sequence of the form $\bsx_{B},\ldots,\bsx_{B+r-1}$  for some $B$.

Then, using the assumptions of the theorem and the triangle inequality for the discrepancy (see \cite[p.~115, Theorem~2.6]{kuinie}), 
we obtain 
\begin{eqnarray*}
\lefteqn{N_jD_{N_j}(X_{j,\ell}) \le}\\
&\le&\sum_{r=1}^v (G_{A,j}(w+u_r)-G_{A,j}(w+u_{r+1})) r D_r (\{\bsx_{w+u_1},\bsx_{w+u_2},\ldots,\bsx_{w+u_r}\})\\ 
&\le& G_{A,j}(w+u_1)  f(v_{A,j})\\ 
&\le& G_j  f(v_j)
.
\end{eqnarray*} 
Using the triangle inequality for the 
discrepancy a second time, we finally obtain 
$$N D_N((\bsx_{g(n)})_{n\ge 0}) \le \sum_{j=0}^{d} a_j G_j f(v_j) 
 \le \sum_{j=0}^d \frac{N_{j+1}}{N_j} G_j  f(v_j).$$
\end{enumerate}
\end{proof}

\section{Indexing by the $q$-ary sum-of-digits function}\label{secsod}

We would now like to show results regarding index-transformed uniformly distributed sequences indexed 
by the $q$-ary sum-of-digits function. We first discuss an application of the general result in 
Theorem~\ref{thmgeneral} (Section~\ref{secsodsdim}) to Halton- and $(t,s)$-sequences, and then show a refined result that applies to 
the particular case of van der Corput-sequences (Section~\ref{secsod1dim}).

\subsection{Results for Halton- and $(t,s)$-sequences}\label{secsodsdim}

Let $q \ge 2$ be an integer and $g(n)=s_q(n)$ the $q$-ary sum-of-digits function. 
For $j \in \NN_0$ choose $N_j=q^j$. Then we have 
$$G_{0,j}(k)=\#\{n\, :\, 0 \le n < q^j, s_q(n)=k\}$$ 
and 
$$(1+x+x^2+\cdots+x^{q-1})^j=\sum_{k \in \NN_0} G_{0,j}(k) x^k,$$
by expanding the polynomial on the left hand side of the latter equation. Hence the sequence 
$(G_{0,j}(k))_{k \in \NN_0}$ is the $j$-fold convolution of the sequence $(\underbrace{1,1,\ldots,1}_{q-\mbox{{\tiny times}}},0,0,\ldots)$, which
implies by \cite[Theorem~1]{OR} that $G_{0,j}(k)$ is unimodal for sufficiently large $j$. 
Since any $n \in \NN_0$ with $A q^j \le n < (A+1)q^j$ 
can be written as $n=n'+A q^j$, where $0 \le n' < q^j$, it follows that $s_q(n)=s_q(n')+s_q(A)$ and hence $G_{A,j}(k)=G_{0,j}(k-s_q(A))$,
where we set $G_{0,j}(k-s_q(A)):=0$ if $k<s_q (A)$. Consequently, $G_{A,j}(k)$ is unimodal for any $A \in \NN_0$ and for sufficiently large $j$.

We recall the following lemma from \cite{DL01}. 

\begin{lemma}[Drmota and Larcher, {\cite[Lemma~1]{DL01}}]\label{lemdrmlar} 
For integers $q \ge 2$, $j \ge 1$, and $0 \le k \le j(q-1)$ 
we have 
$$G_{0,j}(k)=\frac{q^j}{\sqrt{2 \pi j} \sigma_q} 
\exp\left(-\frac{x_{j,k}^2}{2}\right)\left(1+\frac{P_1(x_{j,k})}{\sqrt{j}}
+\frac{P_2(x_{j,k})}{j}\right)+O\left(\frac{q^j}{j^2}\right),$$ 
where $P_1(x)$ and $P_2(x)$ are polynomials, $P_1(x)$ is odd, 
where $x_{j,k}:=\frac{k-\frac{j(q-1)}{2}}{\sigma_q \sqrt{j}}$, 
and where $\sigma_q:=\sqrt{\frac{q^2-1}{12}}$. 
The implied constant in the $O$-notation is uniform for all $k$ and only depends on $q$. 
\end{lemma}

Due to Lemma \ref{lemdrmlar}, there exists some $c_q>0$ such that for sufficiently large $j$ 
we have $G_{A,j}(k) \le c_q q^j/\sqrt{j}$, 
uniformly in $k$ and $A$. Thus we obtain 
\begin{equation}\label{ubsq}
G_j \le c_q \frac{q^j}{\sqrt{j}}
\end{equation}
for sufficiently large $j$.
On the other hand, for $\widetilde{k}=\left\lfloor j\frac{q-1}{2}\right\rfloor$ it follows that 
\begin{equation}\label{lbsq}
\max_{k \in \NN_0} G_{0,j}(k) \ge G_{0,j}(\widetilde{k}) \ge c_q' \frac{q^j}{\sqrt{j}}.
\end{equation}

Furthermore it is clear that $v_0=1$ and $v_j \le q j$ for all $j\in \NN$. 
As an application of Theorem~\ref{thmgeneral}, we obtain the following result.

\begin{theorem}\label{thmsodgen}
Let $X:=(\bsx_{n})_{n \ge 0}$ be an $s$-dimensional sequence such that $ m \widetilde{D}_{m}((\bsx_n)_{n\ge 0})
 \le C (\log m)^s$ for all $m \in \NN$, 
where $C$ may depend on $s$ or on the sequence $X$, but not on $m$. Let $q \ge 2$ be an integer. 
Then there exist $c_q^{(2)},c_q^{(3)} >0$, where $c_q^{(3)}$ may also depend on $s$ and 
$X$, such that  $$\frac{c_q^{(2)}}{\sqrt{\log N}} \le D_N((\bsx_{s_q(n)})_{n\ge 0}) 
\le c_q^{(3)}\frac{(\log \log N)^s}{\sqrt{\log N}}.$$
\end{theorem}

\begin{proof}
Assume that $q^d \le N < q^{d+1}$. Then we obtain from Theorem~\ref{thmgeneral} and 
Equation \eqref{lbsq} that   
$$D_N((\bsx_{s_q(n)})_{n\ge 0}) \ge \frac{c'_q}{N}\frac{q^d}{\sqrt{d}} \ge \frac{c_q^{(2)}}{\sqrt{\log N}}.$$ 
On the other hand, from Theorem~\ref{thmgeneral} and Equation \eqref{ubsq} ,
\begin{eqnarray*}
D_N((\bsx_{s_q(n)})_{n\ge 0}) & \le &  \frac{1}{N} \sum_{j=1}^d q c_q \frac{q^j}{\sqrt{j}} C (\log(q j))^s \\
& \ll_q & (\log d)^s \left(\frac{1}{N} \sum_{1 \le j < d/2} \frac{q^j}{\sqrt{j}}+\frac{1}{N}\sum_{d/2 \le j \le d} \frac{q^j}{\sqrt{j}}\right)\\
& \ll_q & (\log d)^s \left(\frac{\sqrt{\log N}}{\sqrt{N}}+\frac{1}{\sqrt{d}}\right)\\
& \ll_q & \frac{(\log \log N)^s}{\sqrt{\log N}},
\end{eqnarray*}
and the result follows.
\end{proof}

The general lower bound in Theorem \ref{thmsodgen} is best possible with respect to the order of magnitude in $N$. 
This will follow from Theorem~\ref{thmvdc} 
below which deals with van der Corput-sequences.

There are several examples of sequences $X$ which satisfy the conditions in Theorem~\ref{thmsodgen} 
such as Halton- or $(t,s)$-sequences (for a proof of this fact, we refer to Section \ref{appendix} of this paper). We thus obtain the following corollary.

\begin{corollary}\label{cor1}
Let $q \ge 2$ be an integer.
\begin{enumerate}
 \item Let $(\bsx_{n})_{n \ge 0}$ be an 
$s$-dimensional Halton-sequence in pairwise co-prime bases\linebreak $b_1,\ldots,b_s$. 
Then there exist $c^{(2)}_q,c^{(4)}_{q,s,b_1,\ldots,b_s} >0$ such that  
$$\frac{c^{(2)}_q}{\sqrt{\log N}} \le D_N((\bsx_{s_q(n)})_{n=0}^{N-1}) \le c^{(4)}_{q,s,b_1,\ldots,b_s}\frac{(\log \log N)^s}{\sqrt{\log N}}.$$
 \item Let $(\bsx_{n})_{n \ge 0}$ be a $(t,s)$-sequence in base $b$. Then there exist $c^{(2)}_q,c^{(5)}_{q,b,s,t} >0$ 
such that  $$\frac{c^{(2)}_q}{\sqrt{\log N}} \le D_N((\bsx_{s_q(n)})_{n\ge 0}) \le c^{(5)}_{q,b,s,t}\frac{(\log \log N)^s}{\sqrt{\log N}}.$$
\end{enumerate}
\end{corollary}

The result of the first part of Corollary~\ref{cor1} can be improved for the special instance of
van der Corput-sequences, as we will show next.

\subsection{The van der Corput-sequence indexed by the sum-of-digits function}\label{secsod1dim}

The following results are based on a general discrepancy estimate 
which was first presented by Hellekalek \cite{hel2009}. 
The following definitions stem from \cite{hel2009,hel2010,HN}. We refer to these references for further information.

For an integer $b \ge 2$ let $\ZZ_b=\left\{z=\sum_{r=0}^\infty z_r b^r\, : \, z_r \in \{0,\ldots,b-1\}\right\}$ be the set 
of $b$-adic numbers. $\ZZ_b$ forms an abelian group under addition. The set $\NN_0$ is a subset of 
$\mathbb{Z}_b$. The {\it Monna map} $\phi_b:\mathbb{Z}_b \to [0,1)$ is defined by
\begin{equation*}
\phi_b(z) := \sum_{r=0}^\infty \frac{z_r}{b^{r+1}}.
\end{equation*}
Note that the radical inverse function 
$\varphi_b$ is nothing but $\phi_b$ restricted to $\NN_0$. We also define the inverse $\phi_b^+: [0,1)\to \mathbb{Z}_b$ by
\begin{equation*}
\phi_b^+\left(\sum_{r=0}^\infty \frac{x_r}{b^{r+1}}\right):=\sum_{r=0}^\infty x_r b^r,
\end{equation*}
where we always use the finite $b$-adic representation for $b$-adic rationals in $[0,1)$.

For $k \in \mathbb{N}_0$ we can define characters $\chi_k:\mathbb{Z}_b \to \{c \in \mathbb{C}: |c| = 1\}$ of $\mathbb{Z}_b$ by
\begin{equation*}
\chi_k(z) = \exp(2\pi \mathrm{i} \phi_b(k) z).
\end{equation*}
Finally, let $\gamma_k: [0,1) \to \{c \in \mathbb{C}: |c| = 1\}$ where $\gamma_k(x) = \chi_k(\phi_b^+(x))$.

For $b\ge 2$ we put $\rho_b(0)=1$ and $\rho_b(k)=\frac{2}{b^{r+1} \sin(\pi \kappa_r/b)}$ for $k \in \NN$ 
with base $b$ expansion $k=\kappa_0+\kappa_1 b+\cdots +\kappa_r b^r$, $\kappa_r \not=0$.

We have the following general discrepancy bound which is based on the functions $\gamma_k$.
\begin{lemma}\label{lem_hel_gen}
Let $g\in \NN$. For any sequence $(y_n)_{n\ge 0}$ in $[0,1)$ we have
\begin{eqnarray*}
 D_N((y_n)_{n\ge 0}) \le \frac{1}{b^g}+\sum_{k=1}^{b^g -1} \rho_b(k) \left|\frac{1}{N}\sum_{n=0}^{N-1}\gamma_{k}(y_n)\right|.
\end{eqnarray*}
\end{lemma}
\begin{proof}
For the special case of a prime $b$, this result 
was shown by Hellekalek \cite[Theorem~3.6]{hel2009}. Using \cite[Lemma~2.10 and 2.11]{HN} 
it is easy to see that Hellekalek's result can be generalized to the one given in the lemma (cf. \cite{heltalk}).
\end{proof}

We show a discrepancy bound for the van der Corput-sequence indexed by the $q$-ary sum-of-digits function for small values of $q$. 
This result improves on the first part of Corollary~\ref{cor1} for van der Corput-sequences. Moreover, it shows that the general 
lower bound from Theorem~\ref{thmsodgen} is best possible in the order of magnitude in $N$.

\begin{theorem}\label{thmvdc}
Let $b,q \ge 2$ be integers with $q <14$, let $(x_n)_{n\ge 0}$ be the van der Corput-sequence in base $b$ and let $(s_q(n))_{n \ge 0}$ 
be the sequence of the $q$-adic sum-of-digits function. Then we have $$D_N((x_{s_q(n)})_{n \ge 0}) \ll_{b,q} \frac{1}{\sqrt{\log N}}.$$ 
\end{theorem}
\begin{remark}\rm
In view of Theorem~\ref{thmsodgen}, the upper bound in Theorem~\ref{thmvdc} is best possible with respect to the order of magnitude in $N$.
\end{remark}

Before we give the proof of Theorem~\ref{thmvdc}, we need some preparations and auxiliary results. Writing $\mathrm{e}(x) := \exp(2\pi \mathrm{i} x)$ for short, we have 
\begin{eqnarray*}
\frac{1}{N}\sum_{n=0}^{N-1}\gamma_{k}(x_{s_q(n)})= \frac{1}{N}\sum_{n=0}^{N-1} \mathrm{e}\left(s_q(n) \phi_{b}(k)\right)=:T_k(N).
\end{eqnarray*}

\begin{lemma}\label{le1}
Let $b,q \ge 2$ be integers, let $k\in \NN$ and let $(x_n)_{n \ge 0}$ be the van der Corput-sequence in base $b$. 
Then for any $m \in \NN_0$ it is true that 
$$|T_k(q^m)| \le \left(1- \frac{16(q-1)}{q^2}\|\phi_{b}(k)\|^2\right)^{m/2},$$ 
where $\|x\|$ is the distance of a real $x$ to the nearest integer.
\end{lemma}
\begin{proof}
First observe that 
$$T_k(q^m)=\frac{1}{q^m}\sum_{n_0,\ldots,n_{m-1}=0}^{q-1} \mathrm{e}((n_0+\ldots+n_{m-1})\phi_{b}(k)) = (T_k(q))^m.$$ 
We now proceed as in \cite{PS02}. We use the identities $\exp(\icomp x)+\exp(-\icomp x)=2 \cos x$ and $\cos(2x)=1-2 \sin^2 x$ to obtain
\begin{eqnarray*}
|T_k(q)|^2 & = & \frac{1}{q^2} \sum_{n,n'=0}^{q-1} \ee\left((n-n') \phi_{b}(k)\right)\\
& = & \frac{1}{q^2}\left(q+  \sum_{n,n'=0\atop n < n'}^{q-1}\left(\ee\left((n-n') \phi_{b}(k)\right)+\ee\left(-(n-n') \phi_{b}(k)\right)\right)\right)\\
& = & \frac{1}{q^2}\left(q+  2 \sum_{n,n'=0\atop n < n'}^{q-1} \cos\left(2 \pi (n-n')\phi_{b}(k)\right)\right)\\
& = & \frac{1}{q^2}\left(q+  2 \sum_{n,n'=0\atop n < n'}^{q-1} \left(1-2\sin^2\left( \pi(n-n') \phi_{b}(k)\right)\right)\right)\\
& = & 1-\frac{4}{q^2} \sum_{n,n'=0\atop n < n'}^{q-1}\sin^2\left( \pi(n-n') \phi_{b}(k)\right)\\
& \le & 1- \frac{4(q-1)}{q^2} \sin^2(\pi \phi_{b}(k))\\
& \le & 1- \frac{16(q-1)}{q^2}\|\phi_{b}(k)\|^2,
\end{eqnarray*}
 Therefore, 
$$|T_k(q^m)| \le \left(1- \frac{16(q-1)}{q^2}\|\phi_{b}(k)\|^2\right)^{m/2}.$$ 
\end{proof}
We also need the following lemma.
\begin{lemma}\label{le2}
For $k\in\NN$ and any $N \in \NN$ with $q$-adic expansion $N=\sum_{r=0}^R a_r q^r$ we have $$ |T_k(N)| \le \frac{1}{N} \sum_{r=0}^R a_r q^r |T_k(q^r)|.$$
\end{lemma}
\begin{proof}
For $N=\sum_{r=0}^R a_r q^r$,
$$\{0,\ldots,N-1\}=\bigcup_{r=0}^R \{a_Rq^R+\cdots +a_{r+1} q^{r+1},\ldots,a_Rq^R+\cdots +a_{r} q^{r}-1\},$$ 
and hence
\begin{eqnarray*}
N |T_k(N)| & = & \left|\sum_{n=0}^{N-1} \ee\left(s_q(n) \phi_{b}(k)\right)\right| \\ & = & \left| \sum_{r=0}^R \ee\left((a_R+\cdots+a_{r+1})\phi_{b}(k)\right) \sum_{n=0}^{a_r q^r -1} \ee\left(s_q(n) \phi_{b}(k)\right)\right|\\
& \le &  \sum_{r=0}^R \left| \sum_{n=0}^{a_r q^r -1} \ee\left(s_q(n) \phi_{b}(k)\right)\right|\\
& = &  \sum_{r=0}^R \left|\sum_{u=0}^{a_r -1} \ee\left(u \phi_{b}(k)\right) \sum_{n=0}^{q^r -1} \ee\left(s_q(n) \phi_{b}(k)\right)\right|\\
& \le &  \sum_{r=0}^R a_r \left|\sum_{n=0}^{q^r -1} \ee\left(s_q(n) \phi_{b}(k)\right)\right|\\
& = & \sum_{r=0}^R a_r q^r |T_k(q^r)|.
\end{eqnarray*}
\end{proof}
We are now ready to give the proof of Theorem~\ref{thmvdc}.
\begin{proof}
For $k \in\{b^r,\ldots,b^{r+1}-1\}$ we have $\varphi_b(k)=\frac{A_k}{b^{r+1}}$ with 
$A_k \in \{1,\ldots,b^{r+1}-1\}$, where $A_{k_1}\neq A_{k_2}$ for $k_1\neq k_2$.
Hence we obtain from Lemma~\ref{le1}
\begin{eqnarray*}
\sum_{k=1}^{b^g -1} \rho_b(k) |T_k(q^m)| & \le & \sum_{r=0}^{g-1} \frac{2}{b^{r+1} 
\sin(\pi/b)} \sum_{k=b^r}^{b^{r+1}-1} \left(1- \frac{16(q-1)}{q^2}\left\|\frac{A_k}{b^{r+1}}\right\|^2\right)^{m/2}\\
& \le & \sum_{r=0}^{g-1} \frac{2}{b^{r+1} \sin(\pi/b)} \sum_{a=1}^{b^{r+1}-1} 
\left(1- \frac{16(q-1)}{q^2}\left\|\frac{a}{b^{r+1}}\right\|^2\right)^{m/2}.
\end{eqnarray*}
For the inner sum we have 
\begin{eqnarray*}
\lefteqn{\sum_{a=1}^{b^{r+1}-1} \left(1- \frac{16(q-1)}{q^2}\left\|\frac{a}{b^{r+1}}\right\|^2\right)^{m/2}}\\ 
& = & \sum_{1 \le a < b^{r+1}/2}\left(1- \frac{16(q-1)}{q^2}\frac{a^2}{b^{2r+2}}\right)^{m/2}\\
&& \mbox{}+\sum_{b^{r+1}/2 \le a < b^{r+1}}\left(1- \frac{16(q-1)}{q^2}\left(1-\frac{a}{b^{r+1}}\right)^2\right)^{m/2}\\
& = & \frac{1}{b^{m(r+1)}} \sum_{1 \le a < b^{r+1}/2}\left(b^{2 r+2}-\frac{16(q-1)}{q^2} a^2\right)^{m/2}\\
&&\mbox{}+\frac{1}{b^{m(r+1)}}\sum_{b^{r+1}/2 \le a < b^{r+1}}\left(b^{2r+2}-\frac{16(q-1)}{q^2}(b^{r+1}-a)^2\right)^{m/2}\\
& = & \frac{2}{b^{m(r+1)}} \sum_{1 \le a < b^{r+1}/2} \left(b^{2 r+2}-\frac{16(q-1)}{q^2} a^2\right)^{m/2} 
+ \delta(b) \left(1-\frac{4(q-1)}{q^2}\right)^{m/2},
\end{eqnarray*}
where $\delta(b)=0$ when $b$ is odd and $\delta(b)=1$ when $b$ is even. 

The assumption $q <14$ yields $\frac{16(q-1)}{q^2} \ge 1$, and hence 
\begin{eqnarray*}
\sum_{a=1}^{b^{r+1}-1} \left(1- \frac{16(q-1)}{q^2}\left\|\frac{a}{b^{r+1}}\right\|^2\right)^{m/2} & \le &  
\frac{2}{b^{m(r+1)}} \sum_{1 \le a < b^{r+1}/2} \left(b^{2 r+2}- a^2\right)^{m/2} +  \left(\frac{3}{4}\right)^{m/2} \\
& \le & \frac{2}{b^{m(r+1)}} \sum_{u=1}^{b^{2r+2}-1} u^{m/2} +  \left(\frac{3}{4}\right)^{m/2}\\
& \le & \frac{2}{b^{m(r+1)}} \int_1^{b^{2r+2}}u^{m/2} \rd u +  \left(\frac{3}{4}\right)^{m/2}\\
& \ll_{b,q} & \frac{b^{2 r+2}}{m+1} +  \left(\frac{3}{4}\right)^{m/2}
\end{eqnarray*}
with an implied constant depending only on $b$ and $q$. Therefore 
\begin{eqnarray}
\sum_{k=1}^{b^g -1} \rho_b(k) |T_k(q^m)| \ll_{b,q} 
\sum_{r=0}^{g-1} \frac{1}{b^{r+1}} \left(\frac{b^{2(r+1)}}{m+1}+\left(\frac{3}{4}\right)^{m/2}\right) \ll_{b,q}  \frac{b^g}{m+1} \label{eq1},
\end{eqnarray}
again with implied constants depending only on $b$ and $q$.

Assume that $N=\sum_{r=0}^R a_r q^r$. Then, using Lemma~\ref{le2} and \eqref{eq1}, we obtain
\begin{eqnarray*}
\sum_{k=1}^{b^g -1} \rho_b(k) |T_k(N)| & \le & \frac{1}{N} \sum_{m=0}^{R} a_m q^m \sum_{k=1}^{b^g -1} \rho_b(k) |T_k(q^m)|\\
& \ll_{b,q} &  b^{g} \frac{1}{N} \sum_{m=0}^R a_m \frac{q^m}{m+1}.
\end{eqnarray*}
Since 
\begin{eqnarray*}
\frac{1}{N} \sum_{m=0}^R a_m \frac{q^m}{m+1} & \le & \frac{1}{N}\sum_{m=0}^{\lfloor R/2\rfloor } 
a_m q^m +\frac{1}{N} \sum_{m=\lfloor R/2\rfloor +1}^{R} a_m \frac{q^m}{m+1} \\
& \ll_q & \frac{q^{R/2}}{N}+\frac{1}{R} \ll_q  \frac{1}{\log N}
\end{eqnarray*}
we obtain $$\sum_{k=1}^{b^g -1} \rho_b(k) |T_k(N)| \ll_{b,q} \frac{b^{g}}{\log N}.$$ 
From Lemma~\ref{lem_hel_gen} 
it follows that 
$$D_N((x_{s_q(n)})_{n \ge 0}) \ll_{b,q} \frac{1}{b^g}+\frac{b^{g}}{\log N}.$$ 
Choosing $g = \lfloor \log_b \sqrt{\log N} \rfloor$ yields 
$$D_N((x_{s_q(n)})_{n \ge 0}) \ll_{b,q} \frac{1}{\sqrt{\log N}}.$$
\end{proof}

\begin{remark}\rm
We remark that, in principle, the method of proof based on Lemma~\ref{lem_hel_gen} can not only be used for van der Corput-sequences,
but also for Halton-sequences in higher dimensions. However, this leads to a discrepancy bound of order $(\log N)^{-\frac{1}{s+1}}$, 
which is considerably weaker than the one presented in Theorem~\ref{thmsodgen}. 
\end{remark}

\section{Other index-transformations}\label{secother}

In this section, we would now like to discuss index-transformed Halton- and digital $(t,s)$-sequences indexed by a different kind
of sequence than the sum-of-digits function, as, e.g., $(\lfloor n^\alpha\rfloor)_{n\ge 0}$ with $0 <\alpha < 1$. 
The following theorem provides another general result, namely lower and upper bounds on the discrepancy
of sequences indexed by functions which in some sense are ``moderately`` monotonically increasing.
\begin{theorem}\label{thmsqrt}
 Let $A\in\NN_0$ and write $\NN_A:=\{A,A+1,A+2,\ldots\}$. Let $f:\NN_0\To\NN_A$ be surjective and monotonically increasing.
Moreover, define, for $k\in\NN_A$, 
$$F(k):=\#\{n\, :\, n\in\NN_0, f(n)=k\}.$$
Under the assumption that $F(k)$ is monotonically increasing in $k$ for sufficiently large $k$, the following three
assertions hold.
\begin{enumerate}
 \item For an arbitrary sequence $(\bsx_n)_{n\ge 0}$ in $[0,1)^s$ it is true that
 $$\frac{F(f(N)-1)}{N}\le D_N((\bsx_{f(n)})_{n\ge 0}).$$

\item For a Halton-sequence $(\bsx_n)_{n\ge 0}$ in co-prime bases $b_1,\ldots,b_s$,
 $$D_N((\bsx_{f(n)})_{n\ge 0})\le C \frac{2F(f(N-1)+1) (\log N)^s}{N},$$
where $C$ is a constant independent of $N$.

\item For a digital $(t,s)$-sequence $(\bsx_n)_{n\ge 0}$ over $\FF_p$ for prime $p$, 
 $$D_N((\bsx_{f(n)})_{n\ge 0})\le \widetilde{C} p^t \frac{2F(f(N-1)+1) (\log N)^s}{N},$$
where $\widetilde{C}$ is a constant independent of $N$.
\end{enumerate}
\end{theorem}
\begin{proof}
\begin{enumerate}
 \item Let $(\bsx_n)_{n\ge 0}$ be an arbitrary sequence in $[0,1)^s$, and let $f$ and $F$ be as in the theorem. 
If $f(N)=A$, then, due to the properties of $f$, we obtain $F(f(N)-1)=0$, so the lower bound on the discrepancy is trivially fulfilled.

If, on the other hand, $f(N)>A$, then it follows by the surjectivity of $f$ that there exist $n\in\NN_0$ such that $f(n)=f(N)-1$. Furthermore,
whenever $n$ is such that $f(n)=f(N)-1<f(N)$, it follows by the monotonicity of $f$ that $n<N$. Hence, the value $f(N)-1$ occurs $F(f(N)-1)$ times 
among $f(0),\ldots,f(N-1)$, and the point $\bsx_{f(N)-1}$ is attained $F(f(N)-1)$ times in the sequence $\bsx_{f(0)},\ldots,\bsx_{f(N-1)}$. 
The lower bound follows by considering an arbitrarily small interval containing $\bsx_{f(N)-1}$.

\item Without loss of generality, assume $f(0)=0$, i.e., $A=0$. 

Furthermore, it is no loss of generality to assume that $f(1)=1$ and that $F(k)$ is monotonically increasing in $k$ for $k\ge 0$. Indeed, if this is not the case, we can 
disregard a suitable number of initial elements $\bsx_{f(0)},\ldots,\bsx_{f(N_0)}$, without changing the discrepancy of the first $N$ points of the sequence $(\bsx_{f(n)})_{n\ge 0}$ by more than $\frac{N_0}{N}$. 

Let $b_1,\ldots,b_s\ge 2$ be co-prime integers and let $(\bsx_n)_{n\ge0}$ be the corresponding Halton-sequence. For 
estimating the discrepancy, we consider an arbitrary interval 
$$I:=\prod_{i=1}^s [0,\alpha^{(i)})\subseteq [0,1)^s,$$
for some $\alpha^{(1)},\ldots,\alpha^{(s)}\in (0,1]$. For each $i\in\{1,\ldots,s\}$,
choose $m_i$ as the minimal integer such that $N \le b_i^{m_i}$. 
Since $f(N-1)\le N-1$, the $i$-th component $x_{f(n)}^{(i)}$ of a point $\bsx_{f(n)}$, $1\le i\le s$, $0\le n\le N-1$, has
at most $m_i$ non-zero digits in its base $b_i$ representation. From this, it is easily derived that we can restrict ourselves to considering
only $\alpha^{(i)}$ with at most $m_i$ non-zero digits in their base $b_i$ expansion, $1\le i\le s$, as this assumption changes 
$D_N((\bsx_{f(n)})_{n\ge 0})$ by a term of order of at most $N^{-1}$. We can therefore write $I$ as 
the disjoint union of intervals
$$I(j_1,\ldots,j_s):=\prod_{i=1}^s \left[\sum_{r=1}^{j_i -1}\frac{\alpha_r^{(i)}}{b_i^r}, \sum_{r=1}^{j_i}\frac{\alpha_r^{(i)}}{b_i^r}\right),$$
where $ 1\le j_i\le m_i$ for $1\le i\le s$ and the $\alpha_r^{(i)}$ represent the base $b_i$ digits of $\alpha^{(i)}$. Each of the $I(j_1,\ldots,j_s)$ 
can in turn be written as the disjoint union of intervals
$$\prod_{i=1}^s J(j_i,k_i):=\prod_{i=1}^s \left[\sum_{r=1}^{j_i -1}\frac{\alpha_r^{(i)}}{b_i^r}+\frac{k_i}{b_i^{j_i}}, 
 \sum_{r=1}^{j_i-1}\frac{\alpha_r^{(i)}}{b_i^r}+\frac{k_i +1}{b_i^{j_i}}\right),$$
with $1\le j_i\le m_i$ and $0\le k_i\le \alpha_{j_i}^{(i)}-1$. If $\alpha_{j_i}^{(i)}=0$, 
then $J(j_i,k_i)$ is of zero volume containing no points. Hence we can 
restrict ourselves to considering only those $J(j_i,k_i)$ with $\alpha_{j_i}^{(i)}\ge 1$. 

Let now $i\in\{1,\ldots,s\}$ and $v\ge 0$ be fixed. By the construction
principle of the points of the Halton-sequence, we see that
$x_{v}^{(i)}$ is contained in $J(j_i,k_i)$ if and only
if
\begin{equation}\label{eqcondhalton}
\begin{pmatrix}v_0^{(i)}\\ \vdots\\ v_{j_i -2}^{(i)}\\ v_{j_i -1}^{(i)}\end{pmatrix}=
  \begin{pmatrix}\alpha_i^{(1)}\\ \vdots\\ \alpha_i^{(j_i -1)}\\ k_i\end{pmatrix},
\end{equation}
where the $v_r^{(i)}$, $0\le r\le j_i -1$ are the digits of $v$ in base $b_i$. 
Note that \eqref{eqcondhalton} has exactly one
solution $(v_0^{(i)},\ldots,v_{j_i -1}^{(i)})$ modulo $b_i$. Hence we can identify exactly one remainder
$R^{(i)}$ modulo $b_i^{j_i}$, such that $x_v^{(i)}\in
J(j_i,k_i)$ if and only if $v\equiv R^{(i)} \pmod{b_i^{j_i}}$. By
the Chinese Remainder Theorem, there exists exactly one remainder
$R$ modulo $Q:=\prod_{i=1}^s b_i ^{j_i}$ such
that
\[\bsx_v\in \prod_{i=1}^s J (j_i,k_i)\ \mbox{if and only if}\ v\equiv R \pmod{Q}.\]

We now deduce an estimate for the number of points among $\bsx_{f(0)},\ldots,\bsx_{f(N-1)}$ that are contained in 
an interval of the type $\prod_{i=1}^s J (j_i,k_i)$. For short, we denote this number by $A\left(\prod_{i=1}^s J (j_i,k_i)\right)$.

Note that there exists a number $\theta=\theta(R,Q,f(N-1))\in\{0,1\}$ such that $0=f(0)\le R+wQ\le f(N-1)$ if and only if
$w\in\{0,\ldots,\lfloor\frac{f(N-1)}{Q}\rfloor-1+\theta\}$, so
\begin{equation}\label{eqAlower}
A\left(\prod_{i=1}^s J (j_i,k_i)\right)\ge \sum_{w=0}^{\lfloor\frac{f(N-1)}{Q}\rfloor-2+\theta} F(R+wQ)\ge 
\sum_{w=0}^{\lfloor\frac{f(N-1)}{Q}\rfloor-2+\theta} F(wQ),
\end{equation}
where we used the monotonicity of $F$. On the other hand, with the same argument,
\begin{equation}\label{eqAupper}
A\left(\prod_{i=1}^s J (j_i,k_i)\right)\le \sum_{w=0}^{\lfloor\frac{f(N-1)}{Q}\rfloor-1+\theta} F(R+wQ)\le 
\sum_{w=1}^{\lfloor\frac{f(N-1)}{Q}\rfloor+\theta} F(wQ).
\end{equation}

For the following, let $K=\left\lfloor \frac{f(N-1)}{Q}\right\rfloor + \theta$. Let
\[\Sigma_A:= \sum_{r=0}^{(K-1)Q-1} F(r),\]
and note that we can write
\[\Sigma_A = \sum_{w=0}^{K-2} \sum_{r=0}^{Q-1} F(w Q + r) \ge  Q\sum_{w=0}^{K-2} F(w Q)
= Q \sum_{w=0}^{\lfloor\frac{f(N-1)}{Q}\rfloor-2+\theta} F(wQ).
\]
On the other hand, by the definition of $\theta$, 
\[\Sigma_A = \sum_{r=0}^{(\left\lfloor \frac{f(N-1)}{Q}\right\rfloor-1 + \theta)Q-1} F(r)\le\sum_{r=0}^{f(N-1)-1} F(r) \le N-1,\]
from which we conclude that
\begin{equation}\label{eqsumupper}
\sum_{w=0}^{\lfloor\frac{f(N-1)}{Q}\rfloor-2+\theta} F(wQ)\le\frac{N-1}{Q}.
\end{equation}

Moreover, let 
\[\Sigma_B:= \sum_{r=1}^{KQ} F(r),\]
for which we can derive, in the same way as the corresponding estimate for $\Sigma_A$,
\[\Sigma_B \le Q \sum_{w=1}^{\lfloor\frac{f(N-1)}{Q}\rfloor+\theta} F(wQ).\]
Again by the definition of $\theta$, 
\[\Sigma_B = \sum_{r=1}^{(\left\lfloor \frac{f(N-1)}{Q}\right\rfloor + \theta)Q} F(r)\ge\sum_{r=1}^{f(N-1)} F(r) =
  \#\{n\in\NN_0: 0 < f(n) \le f(N-1)\}\ge N-1,\]
where we used that $f(1)=1$ and that $f$ is monotonically increasing. Consequently,
\begin{equation}\label{eqsumlower}
\sum_{w=1}^{\lfloor\frac{f(N-1)}{Q}\rfloor+\theta} F(wQ)\ge\frac{N-1}{Q}.
\end{equation}

Note, furthermore, that
\begin{eqnarray}\label{eqsumdiff}
0\le \sum_{w=1}^{\lfloor\frac{f(N-1)}{Q}\rfloor+\theta} F(wQ)- 
  \sum_{w=0}^{\lfloor\frac{f(N-1)}{Q}\rfloor-2+\theta} F(wQ)
  &\le& F \left(\left(\left\lfloor \frac{f(N-1)}{Q}\right\rfloor-1 + \theta\right)Q\right)\nonumber\\
  &&+F \left(\left(\left\lfloor \frac{f(N-1)}{Q}\right\rfloor + \theta\right)Q\right)\nonumber\\
  &\le& 2F(f(N-1)+1).
\end{eqnarray}
Combining Equations \eqref{eqAlower}, \eqref{eqsumlower}, and \eqref{eqsumdiff}, and noting that $\lambda\left(\prod_{i=1}^s J (j_i,k_i)\right)=\frac{1}{Q}$,  
gives
\begin{eqnarray*}
\frac{1}{N}A\left(\prod_{i=1}^s J (j_i,k_i)\right)- \frac{1}{Q}&\ge& 
\frac{1}{N}\sum_{w=0}^{\lfloor\frac{f(N-1)}{Q}\rfloor-2+\theta} F(wQ)  -\frac{1}{Q}\\
&\ge& \frac{\sum_{w=1}^{\lfloor\frac{f(N-1)}{Q}\rfloor+\theta} F(wQ)- 2F(f(N-1)+1)}{N}  -\frac{1}{Q}\\
&\ge& \frac{- 2F(f(N-1)+1)}{N}+\frac{N-1}{QN}-\frac{1}{Q}\\
&\ge& \frac{- 2F(f(N-1)+1)}{N}-\frac{1}{NQ}.
\end{eqnarray*}
In exactly the same way, using \eqref{eqAupper}, \eqref{eqsumupper}, and \eqref{eqsumdiff}, we get
\[\frac{1}{N}A\left(\prod_{i=1}^s J (j_i,k_i)\right)- \frac{1}{Q}\le \frac{2F(f(N-1)+1)}{N}+\frac{1}{NQ},\]
from which we derive 
\[\abs{\frac{1}{N}A\left(\prod_{i=1}^s J (j_i,k_i)\right)- \frac{1}{Q}}\le \frac{2F(f(N-1)+1)}{N}+\frac{1}{NQ}.\]
Finally, note that, by writing $A(I)$ for the number of points of $(\bsx_{f(n)})_{n=0}^{N-1}$ in $I$, 
\begin{eqnarray*}
\lefteqn{\abs{\frac{A(I)}{N}-\lambda(I)}\le}\\&\le& \sum_{j_1=1}^{m_1}\cdots
\sum_{j_s=1}^{m_s}\sum_{k_1=0}^{\alpha_{j_1}^{(1)}-1}\cdots \sum_{k_s=0}^{\alpha_{j_s}^{(s)}-1}
\abs{\frac{1}{N}A\left(\prod_{i=1}^s J (j_i,k_i)\right)-\lambda\left(\prod_{i=1}^s J (j_i,k_i)\right)}\\
&\le& C \frac{(\log N)^s F(f(N-1)+1)}{N},
\end{eqnarray*}
for a suitably chosen constant $C$, and the result follows.

\item As in Item 2, assume without loss of generality that $f(0)=0$, $f(1)=1$, and that $F(k)$ is monotonically increasing in $k$ for $k\ge 1$. 

Let $p$ be a prime and let $(\bsx_n)_{n\ge0}$ be a digital $(t,s)$-sequence over $\FF_p$. For 
estimating the discrepancy, we consider an arbitrary interval 
$$I:=\prod_{i=1}^s [0,\alpha^{(i)})\subseteq [0,1)^s,$$
for some $\alpha^{(1)},\ldots,\alpha^{(s)}\in (0,1]$. Choose $m$ as the minimal integer such that $N \le p^{m}$. 
By a similar argument as for the case of Halton sequences, we can restrict ourselves to considering only $\alpha^{(i)}$ with 
at most $m$ non-zero digits $\alpha_1^{(i)},\ldots,\alpha_m^{(i)}$ in their base $p$ expansion. Moreover, with the same reasoning as in
the Halton case, we see that we essentially only need to deal with intervals of the form
$$\prod_{i=1}^s J(j_i,k_i):=\prod_{i=1}^s \left[\sum_{r=1}^{j_i -1}\frac{\alpha_r^{(i)}}{p^r}+\frac{k_i}{p^{j_i}}, 
 \sum_{r=1}^{j_i-1}\frac{\alpha_r^{(i)}}{p^r}+\frac{k_i +1}{p^{j_i}}\right),$$
with $1\le j_i\le m$ and $0\le k_i\le \alpha_{j_i}^{(i)}-1$. Again, if $\alpha_{j_i}^{(i)}=0$, 
then $J(j_i,k_i)$ is of zero volume containing no points, so we can 
restrict ourselves to considering only those $J(j_i,k_i)$ with $\alpha_{j_i}^{(i)}\ge 1$.

As for the case of Halton sequences, we would like to derive an upper and a lower bound on the number $A\left(\prod_{i=1}^s J(j_i,k_i)\right)$ of points contained in $\prod_{i=1}^s J(j_i,k_i)$. To this
end, denote the $r$-th row of a generator matrix $C_j$, $1\le j\le s$ of $(\bsx_n)_{n\ge0}$ by $\bsc_r^{(j)}$.

For an integer $v\ge 0$, 
the point $\bsx_v$ is contained in $\prod_{i=1}^s J(j_i,k_i)$ if and only if
\begin{equation}\label{eqlinsystem}
\mathcal{C}\cdot \begin{pmatrix}v_0\\ v_1\\ v_2\\ \vdots \end{pmatrix}=A^{\top},
\end{equation}
where $v_0,v_1, v_2,\ldots $ are the base $p$ digits of $v$, where
\[A:=(\alpha_1^{(1)},\ldots,\alpha_{j_1-1}^{(1)},k_1,
 \alpha_1^{(2)},\ldots,\alpha_{j_2-1}^{(2)},k_2,\ldots\ldots,\alpha_1^{(s)},\ldots,\alpha_{j_s-1}^{(s)},k_s)\in \FF_p^{j_1+\cdots+j_s},\]
and
 \[\mathcal{C}:=\left(\bsc_1^{(1)}, \ldots, \bsc_{j_1}^{(1)}, 
      \bsc_1^{(2)},\ldots, \bsc_{j_2}^{(2)}, 
      \ldots \ldots,
      \bsc_1^{(s)}, \ldots, \bsc_{j_s}^{(s)}\right)^\top \in \FF_p^{(j_1+\cdots+j_s)\times \NN}.\] 
Let now $Q:=p^{j_1+\cdots + j_s +t}$, let $w\in\NN_0$ and consider those $v\ge 0$ with $wQ \le v \le (w+1)Q -1$. For these $v$, the first $j_1+j_2+\cdots + j_s +t$ digits in their 
base $p$ expansion vary, while all the other digits are fixed. Hence we can write \eqref{eqlinsystem} as
\[
D_1\cdot \begin{pmatrix}v_0\\ v_1\\ \vdots\\ v_{j_1+\cdots j_s +t} \end{pmatrix}
+ D_2 \cdot \begin{pmatrix}v_{j_1+\cdots+j_s+t+1} \\ v_{j_1+\cdots+j_s+t+2}\\ \vdots\\ \end{pmatrix} =A^{\top},
\]
where $\mathcal{C}=(D_1|D_2)$ and where $D_1$ is an $(j_1+\cdots +j_s)\times (j_1 +\cdots + j_s +t)$-matrix and $D_2$ is an $(j_1+\cdots + j_s)\times\NN$-matrix over $\FF_p$.

Due to the fact that $(\bsx_n)_{n\ge 0}$ is a digital $(t,s)$-sequence, it follows that $D_1$ has full rank, and hence there are exactly $p^t$ values $v$ in $\{wQ, wQ+1,\ldots, (w+1)Q-1\}$ such
that $\bsx_v$ is contained in $\prod_{i=1}^s J(j_i,k_i)$. 

Now note again that there exists a number $\theta=\theta(Q,f(N-1))\in\{0,1\}$ such that $0=f(0)\le wQ\le f(N-1)$ if and only if
$w\in\{0,\ldots,\lfloor\frac{f(N-1)}{Q}\rfloor-1+\theta\}$. By our observations above, for each of these $w\in\{0,\ldots,\lfloor\frac{f(N-1)}{Q}\rfloor-1+\theta\}$ there exist $p^t$ integers $R_{w,1},\ldots,R_{w,p^t}\in\{0,\ldots,Q-1\}$ such that exactly the points $\bsx_{R_{w,1}+wQ},\ldots,\bsx_{R_{w,p^t}+wQ}$ among $\bsx_{wQ},\bsx_{wQ+1},\ldots,\bsx_{(w+1)Q-1}$ are contained in
$\prod_{i=1}^s J(j_i,k_i)$. Therefore, we can estimate 
\begin{equation}\label{eqAlower2}
A\left(\prod_{i=1}^s J (j_i,k_i)\right)\ge \sum_{w=0}^{\lfloor\frac{f(N-1)}{Q}\rfloor-2+\theta}\sum_{z=1}^{p^t} F(R_{w,z}+wQ)\ge 
p^t\sum_{w=0}^{\lfloor\frac{f(N-1)}{Q}\rfloor-2+\theta} F(wQ),
\end{equation}
and
\begin{equation}\label{eqAupper2}
A\left(\prod_{i=1}^s J (j_i,k_i)\right)\le \sum_{w=0}^{\lfloor\frac{f(N-1)}{Q}\rfloor-1+\theta}\sum_{z=1}^{p^t} F(R_{w,z}+wQ)\le 
p^t\sum_{w=1}^{\lfloor\frac{f(N-1)}{Q}\rfloor+\theta} F(wQ).
\end{equation}
In exactly the same way as for a Halton sequence, we obtain, by noting that $\lambda\left(\prod_{i=1}^s J (j_i,k_i)\right)=\frac1{p^{l_1+\cdots +l_s}}=\frac{p^t}{Q}$, 
\[\abs{\frac{1}{N}A\left(\prod_{i=1}^s J (j_i,k_i)\right)- \frac{1}{Q}}\le \frac{p^t2F(f(N-1)+1)}{N}+\frac{p^t}{NQ},\]
and the result follows.
\end{enumerate}
\end{proof}

Examples of functions $f$ and $F$ satisfying the assumptions of Theorem~\ref{thmsqrt} are obtained as follows. Let $g:\RR_0^{+}\To\RR_0^{+}$
be a function that is twice differentiable on $(0,\infty)$, with $g'(x)>0$ and $g''(x)<0$ for $x\in(0,\infty)$. Moreover, define
$f(n):=\lfloor g(n) \rfloor$ for $n\in\NN$. It then easily follows that $f$ and $F$ indeed fulfill the assumptions of the theorem
and we obtain 
\begin{equation}\label{eqbigF}
 F(k+1)=\left\lceil g^{-1} (k+1) \right\rceil - \left\lceil g^{-1} (k) \right\rceil.
\end{equation}
We thus obtain the following exemplary corollary to Theorem~\ref{thmsqrt}.
\begin{corollary}
Let $\alpha\in (0,1)$.Then the following assertions hold.
\begin{enumerate}
\item For a Halton-sequence $(\bsx_n)_{n\ge 0}$ in co-prime bases $b_1,\ldots,b_s$,
 $$\overline{C}_1\frac{1}{N^{\alpha}}\le D_N((\bsx_{\lfloor n^{\alpha}\rfloor})_{n\ge 0})\le \overline{C}_2 \frac{(\log N)^s}{N^{\alpha}},$$
where $\overline{C}_1$, $\overline{C}_2$ are constants that depend on the sequence and on $\alpha$, but are independent of $N$.

\item For a digital $(t,s)$-sequence $(\bsx_n)_{n\ge 0}$ over $\ZZ_p$ for prime $p$, 
 $$\overline{\overline{C}}_1\frac{1}{N^{\alpha}}\le D_N((\bsx_{\lfloor n^{\alpha}\rfloor})_{n\ge 0})\le \overline{\overline{C}}_2 \frac{(\log N)^s}{N^{\alpha}},$$
where $\overline{\overline{C}}_1$, $\overline{\overline{C}}_2$ are constants that depend on the sequence and on $\alpha$, but are independent of $N$.
\end{enumerate}
\end{corollary}
\begin{proof}
The result follows by combining Theorem~\ref{thmsodgen} with the observation that $$c'_{\alpha}k^{\frac{1}{\alpha}-1} \le F(k)\le c_{\alpha} k^{\frac{1}{\alpha}-1},$$ with constants $c'_{\alpha},c_{\alpha}>0$ that depend on $\alpha$, but not on $k$.
\end{proof}

\section{Appendix: Uniform discrepancy}\label{appendix}

In Corollary~\ref{cor1} we implicitly used the fact that $(t,s)$-sequences in base $b$ 
as well as Halton-sequences in pairwise co-prime bases $b_1,\ldots,b_s$ have uniform discrepancy of order $(\log N)^s/N$. 
Since we are not aware of a proof of these facts in the existing literature, we provide one here.

\subsection{Uniform discrepancy of $(t,s)$-sequences in base $b$}

Assume that $\Delta_b(t,m,s)$ is a number for which $$b^m D_{b^m}(\cP) \le \Delta_b(t,m,s)$$ holds for the discrepancy of any $(t,m,s)$-net $\cP$ in base $b$.

\begin{theorem}\label{thmA1}
Let $(\bsx_n)_{n \ge 0}$ be a $(t,s)$-sequence in base $b$. Then we have 
$$N \widetilde{D}_N((\bsx_n)_{n \ge 0}) \le (2b-1)\left(t b^t + \sum_{m=t}^{\lfloor \log_b N\rfloor} \Delta_b(t,m,s)\right).$$ 
\end{theorem}
\begin{proof}
Let $k \in \NN_0$. We show that 
$$N D_N((\bsx_{n+k})_{n \ge 0})  \le (2b-1)\left(t b^t+\sum_{m=t}^{\lfloor \log_b N\rfloor} \Delta_b(t,m,s)\right)$$ 
uniformly in $k \in \NN_0$.

For $N< b^t$, the assertion follows trivially by $N D_N((\bsx_{n+k})_{n \ge 0})  \le N$.

Let now $N \in \NN$, $N\ge b^t$  with $b$-adic expansion $N=a_rb^r+a_{r-1} b^{r-1}+\cdots+a_1 b+a_0$ where $a_j \in \{0,\ldots,b-1\}$ 
for $0 \le j \le r$ and $a_r \not=0$ (note that $r\ge t$). For given $k\in\NN_0$, choose $\ell \in \NN$ such that $(\ell-1) b^r \le k < \ell b^r$. 
Then we can write 
$$k=\ell b^r-(d_{r-1} b^{r-1}+\cdots+d_1 b+d_0)-1$$ 
with some $d_j \in \{0,\ldots,b-1\}$ for $0 \le j \le r-1$, and
$$k=(\ell-1) b^r+\kappa_{r-1}b^{r-1}+\cdots +\kappa_1 b+\kappa_0$$ 
with some $\kappa_j \in \{0,\ldots,b-1\}$ for $0 \le j \le r-1$. Note that therefore $d_j+\kappa_j=(b-1)$ for $0 \le j < r$.

We split up the point set $\cP_{k,N}:=\{\bsx_n\, : \, k \le n \le k+N-1\}$ in the following way:
\begin{eqnarray*}
\cP_{k,N} &=&\bigcup_{1 \le d \le d_0+1} \cP'_{0,d} \cup \bigcup_{1 \le m \le t-1\atop 1 \le d \le d_m}\cP'_{m,d} \cup \bigcup_{t \le m \le r-1\atop 1 \le d \le d_m}\cP'_{m,d}\\
&& \cup \bigcup_{0 \le a \le a_r -2} \cP''_a \cup \bigcup_{0 \le m \le t-1 \atop 0 \le x \le a_m+\kappa_m -1} \cP'''_{m,x} 
\cup \bigcup_{t \le m \le r-1 \atop 0 \le x \le a_m+\kappa_m -1} \cP'''_{m,x},
\end{eqnarray*}
where 
\begin{eqnarray*}
\cP'_{m,d} & := & \{\bsx_{\ell b^r-d_{r-1}b^{r-1}-\cdots-d_{m+1} b^{m+1}-d b^m+j} \, : \, 0 \le j < b^m\},\\
\cP''_a & := &  \{\bsx_{\ell b^r+a b^r+j}\, : \, 0 \le j < b^r\},\\
\cP'''_{m,x} & := & \{\bsx_{(\ell+a_r-1)b^r+(\kappa_{r-1}+a_{r-1}) b^{r-1}+\cdots +(\kappa_{m+1}+a_{m+1}) b^{m+1}+x b^m+j}\; : \; 0 \le j < b^m\}.
\end{eqnarray*}
For $m\le t-1$, we can bound the discrepancy of $\cP'_{m,d}$ and $\cP'''_{m,x}$, respectively, by the trivial bound 1. For $m\ge t$,
the point sets $\cP'_{m,d}$ and $\cP'''_{m,x}$ are $(t,m,s)$-nets in base $b$, and the $\cP''_a$ are $(t,r,s)$-nets in base $b$. 
From the triangle inequality for the discrepancy we obtain 
\begin{eqnarray*}
N D_N(\cP_{k,N}) & \le &  (d_0+a_0+\kappa_0+1) b^0 +  \sum_{m=1}^{t-1} (d_m+a_m+\kappa_m) b^m\\&& + \sum_{m=t}^{r-1} (d_m+a_m+\kappa_m) \Delta_b(t,m,s)
+ \max(a_r-2,0) \Delta_b(t,r,s)\\
& \le & (2 b-1)  + (2 b-2)\left((t-1)b^t+ \sum_{m=t}^{r-1} \Delta_b(t,m,s)\right)\\ 
&&+ \max(b-3,0) \Delta_b(t,r,s)\\
& \le & (2b-1)\left(tb^t+ \sum_{m=t}^r \Delta_b(t,m,s)\right)
\end{eqnarray*} 
and the result follows, since $r=\lfloor \log_b N\rfloor$.  
\end{proof}

\begin{corollary}
Let $(\bsx_n)_{n \ge 0}$ be a $(t,s)$-sequence in base $b$. Then we have $$N \widetilde{D}_N((\bsx_n)_{n \ge 0}) \ll_{s,b} b^t (\log N)^s.$$ 
\end{corollary}

\begin{proof}
The result follows from Theorem~\ref{thmA1} together with the fact that $$\Delta_b(t,m,s) \ll_{s,b} b^t m^{s-1}$$ for $m \ge t$ (see, for example, \cite{DP10,niesiam}).
\end{proof}

\subsection{Uniform discrepancy of Halton-sequences}

\begin{theorem}
Let $(\bsx)_{n \ge 0}$ be a Halton-sequence in pairwise co-prime bases $b_1,\ldots,b_s$. 
Then we have 
$$N \widetilde{D}_N((\bsx_n)_{n \ge 0})= \frac{1}{s!} \prod_{j=1}^s\left(\frac{\lfloor b_j/2\rfloor \log N}{\log b_j}+s\right)+O((\log N)^{s-1}),$$
where the implied constant depends on $b_1,\ldots,b_s$ and $s$. 
\end{theorem}

\begin{proof}
The result follows from an adaption of the proof of \cite[Theorem~3.36]{DP10}. 
Note that \cite[Lemma~3.37]{DP10} also holds true for 
$A(J,k,N,\mathcal{S}):=\#\{n \in \NN\, : \, k \le n < k+N \mbox{ and } \bsx_n \in J\}$ instead of $A(J,N,\mathcal{S}):=A(J,0,N,\mathcal{S})$.
The rest of the proof of \cite[Theorem~3.36]{DP10} remains unchanged. 
\end{proof}

\section*{Acknowledgements}

The authors would like to thank M.~Drmota for valuable suggestions and remarks.

\begin{small}
\noindent\textbf{Authors' address:}\\
\noindent Peter Kritzer, Gerhard Larcher, Friedrich Pillichshammer\\
Institut f\"{u}r Finanzmathematik, Johannes Kepler Universit\"{a}t Linz, Altenbergerstr.~69, 4040 Linz, Austria\\
E-mail: \texttt{peter.kritzer@jku.at},\\ \texttt{gerhard.larcher@jku.at},\\ 
\texttt{friedrich.pillichshammer@jku.at}
\end{small}


\begin{thebibliography}{99}

\bibitem{A04} Atanassov,~E.I.: On the discrepancy of the Halton sequences. Math. Balkanica (N.S.) 18: 15--32, 2004.

\bibitem{bej} B\'{e}jian,~R.: Minoration de la discr\'{e}pance d'une suite quelconque sur $T$. Acta Arith. 41: 185--202, 1982.

\bibitem{befa} B\'{e}jian,~R. and Faure,~H.: Discr\'{e}pance de la suite de van der Corput.  C. R. Acad. Sci., Paris, S\'{e}r. A 285: 313--316, 1977.

\bibitem{BLV08} Bilyk,~D., Lacey,~M.T., and Vagharshakyan, A.: On the small ball inequality in all dimensions. J. Funct. Anal. 254: 2470--2502, 2008.

\bibitem{C80} Coquet,~C.: Sur certaines suites uniform\'{e}ment \'{e}quir\'{e}parties modulo 1, Acta Arith. 36: 157--162, 1980.

\bibitem{DP10} Dick,~J. and Pillichshammer,~F.: {\it Digital Nets and 
Sequences---Discrepancy Theory and Quasi-Monte Carlo Integration.} Cambridge University Press, Cambridge, 2010.

\bibitem{D98} Drmota,~M.: $q$-additive functions and well distribution modulo 1. Demonstratio Math. 30: 883--896, 1998.

\bibitem{DL01} Drmota,~M. and Larcher,~G.: The sum-of-digits-function and uniform distribution modulo 1. J. Number Theory 89: 65--96, 2001.

\bibitem{dt} Drmota,~M. and Tichy,~R.F.: {\it Sequences, Discrepancies and Applications.} Lecture Notes in Mathematics 1651, Springer-Verlag, Berlin, 1997.

\bibitem{fau80} Faure, H.: Suites \`{a} faible discr\'{e}pance dans $T^s$. Publ. D\'{e}p. Math., Universit\'{e} de Limoges, Limoges, France, 1980.

\bibitem{FL} Faure,~H. and Lemieux,~C.: Improved Halton sequences and discrepancy bounds. Monte Carlo Models and Applications 16: 231--250, 2010.

\bibitem{gel} Gel'fond,~A.O.: Sur les nombres qui ont des propri\'{e}t\'{e}s
additives et multiplicatives donn\'{e}es. Acta Arith. 13: 259--265, 1968.

\bibitem{halton} Halton, J.H.: On the efficiency of certain quasi-random sequences of 
points in evaluating multi-dimensional integrals. Numer. Math. 2: 84--90, 1960. Erratum, ibid. 2: p. 196, 1960. 

\bibitem{hel2009} Hellekalek,~P.: A general discrepancy estimate based on $p$-adic arithmetics. Acta Arith.~139: 117--129, 2009.

\bibitem{hel2010} Hellekalek,~P.: A notion of diaphony based on $p$-adic arithmetic. Acta Arith.~145: 273--284, 2010.

\bibitem{heltalk} Hellekalek,~P.: Assessing randomness: tools from $b$-adic analysis. Talk at the Tenth International
Conference on Monte Carlo and Quasi-Monte Carlo Methods in Scientific Computing. Sydney, Feb.~14 2012.

\bibitem{HN} Hellekalek,~P. and Niederreiter,~H.: 
Constructions of uniformly distributed sequences using the $b$-adic method. Unif. Distrib. Theory~6: 185--200, 2011.

\bibitem{HKLP}  Hofer,~R., Kritzer,~P., Larcher,~G., and Pillichshammer,~F.: 
Distribution properties of generalized van der Corput-Halton sequences and their subsequences.  Int. J. Number Theory~5: 719--746, 2009.

\bibitem{huawang} Hua, L. K. and Wang, Y.: {\it Applications of number theory to numerical analysis.} Springer-Verlag, Berlin-New York, 1981.

\bibitem{kuinie} Kuipers,~L. and Niederreiter,~H.: {\it Uniform Distribution of
Sequences}. John Wiley, New York, 1974; reprint, Dover Publications, Mineola, NY, 2006.

\bibitem{matou} Matou\v{s}ek,~J.: {\it Geometric discrepancy. An illustrated guide.} Algorithms and Combinatorics, 18. Springer-Verlag, Berlin, 1999.

\bibitem{meijer} Meijer,~H.~G.: The discrepancy of a $g$-adic sequence. Indag. Math. 30: 54--66, 1968.

\bibitem{N87} Niederreiter,~H.: Point sets and sequences with small discrepancy. Monatsh. Math.~104: 273--337, 1987.

\bibitem{niesiam} Niederreiter,~H.: {\it Random Number Generation
and Quasi-Monte Carlo Methods.} SIAM, Philadelphia, 1992.

\bibitem{OR} Odlyzko,~A.~M. and Richmond,~L.~B.: On the unimodality of high convolutions of discrete distributions. Ann. Probab.~13: 299-306, 1985.

\bibitem{P12} Pillichshammer,~F.: On the discrepancy of the van der Corput sequence indexed by Fibonacci numbers. Fibonacci Quart. 50: 235--238, 2012.


\bibitem{PS02} Puchta,~J.~Ch. and Spilker,~J.: Altes und Neues zur Quersumme. Math. Semesterber.~49:209--226, 2002.

\bibitem{roth1} Roth,~K.~F.: On irregularities of distribution. Mathematika 1: 73--79, 1954.

\bibitem{Schm72distrib} Schmidt,~W.~M.: Irregularities of distribution VII. Acta Arith. 21:  45--50, 1972.

\bibitem{TTa} Tichy,~R.~F. and Turnwald, G.: Gleichm\"assige Diskrepanzabsch\"atzung f\"ur Ziffernsummen. Anz. \"Osterreich. Akad. Wiss. Math.-Natur. Kl.  123  (1986): 17--21, 1987.

\bibitem{TT} Tichy,~R.~F. and Turnwald,~G.: On the discrepancy of some special sequences. J. Number Theory 26: 351--366, 1987.
\end{thebibliography}
\end{document}